\documentclass[11pt]{article}

\usepackage[margin=1in]{geometry}
\usepackage{graphicx} 
\usepackage{hyperref}   
\hypersetup{
	colorlinks=true,
	linkcolor=blue,
	citecolor=blue,
	urlcolor=blue,
}
\usepackage[square,numbers]{natbib}
\usepackage{xcolor}
\usepackage{amsthm}
\usepackage{algorithm}
\usepackage{algpseudocode}
\usepackage{makecell}     
\usepackage{nicematrix}
\usepackage{graphicx}
\usepackage{amsfonts}
\usepackage{siunitx}
\usepackage{tabularray}
\usepackage{amsmath,mathtools,amssymb,amsfonts}

\newcommand{\norm}[1]{\left\lVert#1\right\rVert}
\newcommand{\xddots}{%
  \raise 4pt \hbox {.}
  \mkern 6mu
  \raise 1pt \hbox {.}
  \mkern 6mu
  \raise -2pt \hbox {.}
}

\newtheorem{theorem}{\bf{Theorem}}

\newtheorem{remark}{\bf{Remark}}
\newtheorem{definition}{Definition}
\newcommand{\keywords}[1]{\textbf{Keywords:} #1}

\newcommand{\assumption}[1]{\textbf{Assumption} #1}

\title{Disturbance-Adaptive Finite-Time Control of Three-Phase Rectifiers}
\author{Koto Omiloli, Satish Vedula, Ayobami~Olajube and Olugbenga Moses Anubi}

\date{}

\begin{document}

\maketitle
\begin{center}
Department of Electrical and Computer Engineering\\
Center for Advanced Power Systems, Florida State University\\
E-mail: \{kao23a, svedula, aolajube, oanubi\}@fsu.edu
\end{center}

\section*{Abstract}
Three-phase AC--DC rectifiers are fundamental components in modern power electronics systems, yet achieving rapid voltage regulation and precise current tracking under load and grid disturbances remains challenging due to nonlinear dynamics and measurement uncertainties. This paper presents a finite-time control method for three-phase AC--DC rectifiers that achieves millisecond-scale regulation of DC-link voltage and grid currents under varying conditions. The proposed design employs a transformed augmented error-state dynamics model, extending the voltage dynamics to a two-state system to construct an adaptive sliding surface that guarantees fast finite-time convergence. A nonlinear sliding-mode voltage regulator with an online disturbance estimator ensures rapid and robust voltage tracking, while a fast current controller achieves finite-time $dq$-axis current tracking with minimal chattering. Theoretical results establishes finite-time stability and provides explicit gain selection conditions. Simulation results demonstrate up to 99.40\% and 87.5\% reductions in voltage and current convergence times, respectively, compared to conventional robust controllers. Laboratory experiments further validate the approach, showing 33.33\% lower voltage ripple, 33.33\% faster rise time, and 32.43\% reduced steady-state error relative to a recent method. These results confirm improvements in transient performance, convergence, and overall system stability, highlighting the method’s practical applicability for high-performance rectifier control.

\keywords{Chattering, Convergence, Finite-
time, Rectifier, Sliding Mode, Surface, Terminal.}

\section{Introduction}

Power converters are essential for the exchange, processing, and control of electrical energy, enabling transformation between alternating current (AC) and direct current (DC) supplies. Modern converters often include nonlinear elements, such as high-frequency switching power semiconductors, and reactive components, such as capacitors and inductors, for energy storage and filtering. In microgrids, these converters play a critical role in ensuring power quality and delivering electrical energy in suitable forms for loads \cite{jamil2009microgrid}. As converter penetration increases in power networks, high-performance control solutions are required to stabilize voltage and current states to improve dynamic response.

Among power converters, AC–DC rectifiers are widely used in industrial equipment, electronic devices, battery chargers, and motor drives. However, without robust control, these systems remain susceptible to disturbances such as voltage sags, load variations, and parameter uncertainties \cite{Shiming}. Conventional linear proportional-integral (PI) controllers are often employed due to their simplicity and effectiveness, including in cascaded or proportional-resonant architectures \cite{guler2023simplified}. Nevertheless, such controllers may struggle in fast-changing environments, exhibiting slow response or steady-state error under abrupt disturbances.

To improve dynamic performance, nonlinear control approaches have been explored. Feedback-based controllers have been proposed for AC-connected DC microgrids \cite{iovine2016nonlinear}, and super-capacitor-assisted schemes have been used to maintain voltage stability. While effective under small disturbances, these methods can exhibit significant delays or sensitivity to large-scale perturbations \cite{jamil2009microgrid}. Sliding-mode control (SMC) has also been widely applied due to its robustness against external disturbances and parameter variations \cite{feng2002non}. Advanced designs integrate extended state observers or super-twisting algorithms (STA) for current and voltage regulation \cite{liu2020sliding}, yet they often require precise parameter tuning, limiting practical implementation.

Nonlinear disturbance-observer-based multiple-surface controllers and finite-time backstepping strategies have further improved convergence and robustness \cite{zolfaghari2019power, ding2024fixed} of converters. However, these methods either rely heavily on accurate models, require specific initial conditions, or achieve only asymptotic convergence, which may not be sufficient for rectifier control applications requiring millisecond-scale response. Model Predictive Control (MPC) has been employed to optimize rectifier operation, minimizing voltage error, current ripple, and power losses \cite{olajube2024decentralized}. Despite its advantages, MPC often imposes high computational burdens, motivating research into computationally efficient implementations \cite{xu2021computationally, rath2024reduced}.

Recent studies have also explored hybrid, cooperative, and intelligent control strategies that combine the strengths of multiple control paradigms to enhance system robustness and steady-state accuracy. Examples include dissipative-based control for energy-shaping and damping injection \cite{anubi2022robust}, non-singular terminal sliding mode control for finite-time state regulation without singularity issues, and neural-network-based controllers for online adaptation to nonlinear and uncertain system dynamics \cite{li2013artificial,  ding2024passivity}. These approaches have demonstrated improved robustness to parameter variations and better steady-state voltage regulation compared to classical linear controllers. However, despite their effectiveness, most of these methods rely on asymptotic, which may not guarantee rapid transient recovery under abrupt load changes or grid frequency fluctuations. Furthermore, data-driven and hybrid designs often involve high computational complexity, sensitivity to network training quality, or require precise model knowledge to achieve consistent performance. Consequently, achieving ultra-fast convergence of rectifier states while maintaining robustness, simplicity, and adaptability remains an open and technically demanding problem in the control of power converters.

Although power-loss minimization is often a key control objective in converter systems, the ability to ensure fast convergence of voltage and current states is equally, if not more, critical to achieving high power quality and stable operation. In grid-connected rectifiers, a slow transient response can lead to extended voltage dips, current overshoot, and oscillations that adversely affect both the DC-link and the upstream AC grid. Moreover, delayed settling times amplify harmonic distortion and reactive power fluctuations, indirectly increasing conduction and switching losses in semiconductor devices. Therefore, designing controllers that enable rapid voltage and current regulation is essential for minimizing transient energy exchange, maintaining tight voltage stability under load variations, and preventing overcurrent stress on hardware components.

Recent advances in MPC, passivity-based control, and SMC have improved rectifier regulation, yet few works explicitly ensure fast finite-time convergence under load uncertainties. Notably, \cite{fu2023finite} achieved ~5 ms convergence for three-phase rectifiers, but the controller relied on fixed gains, limiting robustness to time-varying loads and grid fluctuations. To overcome these limitations, this paper introduces a disturbance-adaptive finite-time control framework for three-phase rectifiers. The method guarantees fast convergence of voltage and current states without ideal-grid assumptions or extensive tuning. An adaptive law embedded in the finite-time sliding surface provides online compensation for unknown load disturbances while keeping current chattering low. The design builds on a transformed augmented error-state model that extends the voltage dynamics to a two-state form, enabling an adaptive sliding surface that ensures finite-time voltage convergence. A complementary fast current regulator achieves finite-time grid-current tracking with smooth transients.

The proposed control scheme achieves millisecond-scale convergence of rectifier states, leading to faster voltage recovery, improved transient stability, and stronger disturbance resilience compared to conventional robust controllers. The main contributions of this work are as follows:

\begin{enumerate}
\item A disturbance-adaptive reaching law for voltage regulation, enabling near-ideal DC voltage tracking within milliseconds.
\item Finite-time convergence analysis and explicit gain selection guidelines for the voltage controller.
\item Design of a fast current regulator with explicitly defined gains, ensuring rapid convergence of grid currents while suppressing chattering.
\item Experimental validation on a physical rectifier circuit, confirming the proposed method’s superior performance.
\end{enumerate}

\section{Notations and Definitions}  
The sets of real and positive real numbers are denoted by $\mathbb{R}$ and $\mathbb{R}_+$, respectively. Scalars are represented by lowercase letters ($x \in \mathbb{R}$) and vectors by bold lowercase letters ($\mathbf{x} \in \mathbb{R}^n$). The zero vector in $\mathbb{R}^n$ is denoted by $\mathbf{0}_n$.  

The skew-symmetric matrix $
J =
\begin{bmatrix}
0 & 1 \\ -1 & 0
\end{bmatrix}
$. For $\mathbf{x} \in \mathbb{R}^n$ and $\alpha \in \mathbb{R}$, $\mathbf{x}^\alpha$ denotes element-wise exponentiation. The Euclidean norm of $\mathbf{x}$ is given by  $
\|\mathbf{x}\|_2 = \sqrt{\sum_{i=1}^{n} |x_i|^2},
$ and $|x|$ denotes the absolute value of $x \in \mathbb{R}$.

\begin{definition}[Signum function]
For $x \in \mathbb{R}$,
\[
\textsf{sgn}(x) =
\begin{cases}
-1, & x < 0,\\
0,  & x = 0,\\
1,  & x > 0.
\end{cases}
\]
\end{definition}

\begin{definition}[Saturation function]
The saturation function $\textsf{sat}: \mathbb{R}^n \to \mathbb{R}^n$ is defined component-wise as
\[
\textsf{sat}(s)_i =
\begin{cases}
s_i, & |s_i| \leq 1,\\
\textsf{sgn}(s_i), & |s_i| > 1,
\end{cases}
\quad i = 1, \dots, n.
\]
\end{definition}

$\mathcal{L}_\infty$ denotes the set of measurable function $f: \mathbb{R}_+ \to \mathbb{R}$ bounded as follows:
\[
\|f\|_\infty = \underset{t \in \mathbb{R}_+}{\textsf{ess sup}}\, |f(t)| < \infty.
\]

 \section{Problem Formulation}
The control objective is to regulate both the AC currents and DC-link voltage of the converter while ensuring finite-time convergence and robustness against disturbances. The control is implemented in the $dq$ reference frame, where sinusoidal AC quantities are transformed into DC equivalents, simplifying regulation \cite{liu2020sliding}.  

The converter dynamics in the $dq$ frame are expressed as
\begin{equation}
\begin{aligned}\label{power_converter}
l \frac{d\mathbf{i}}{dt} &= -r \mathbf{i} + \omega_g l J \mathbf{i} + \mathbf{v} - \mathbf{u}v_{dc}, \\
c \frac{dv_{dc}}{dt} &= \mathbf{u}^\top \mathbf{i} - i_l,
\end{aligned}
\end{equation}
where $\mathbf{i} = [\,i_d\; i_q\,]^\top \in \mathbb{R}^2$ and $\mathbf{v} = [\,v_d\; v_q\,]^\top \in \mathbb{R}^2$ are the converter currents and voltages in the $dq$ frame, respectively. The DC-link voltage is $v_{dc} \in \mathbb{R}$, and $\mathbf{u} = [\,u_d\; u_q\,]^\top \in \mathbb{R}^2$ denotes the control input. The parameters $l$, $r$, and $c$ represent the line inductance, resistance, and DC-link capacitance, while $\omega_g \in \mathbb{R}_+$ denotes the grid angular frequency.

The following assumptions are made in this work to simplify the analysis and control design:\\
\assumption{\textbf{1:}
The AC source voltages and currents are assumed to be balanced.} \\
\begin{assumption}\textbf{2:}
 The control design is based on the assumption of complete knowledge of the rectifier dynamics.
\end{assumption} \\
\begin{assumption}\label{disturbance_assumption}\textbf{3:}
   The resistive load disturbance, $\rho(t)$ is piece-wise constant and satisfies $|\rho(t)|\le \delta$ and $|\dot{\rho}(t)|\le \epsilon$ for some known $\delta>0$ and $\epsilon>0$. 
\end{assumption}

\section{CONTROL DEVELOPMENT}
The proposed control structure consists of two cascaded loops: an outer voltage control loop and an inner current control loop. The voltage control loop regulates the DC-link voltage $v_{dc}$ to its reference $v_{dc}^*$, generating the reference current $i_d^*$ for the inner loop, as shown in Fig.~\ref{AC-DC}. Both loops are driven by nonlinear regulators based on state feedback.

\begin{figure}[h]
    \centering \includegraphics[width=0.6\textwidth]{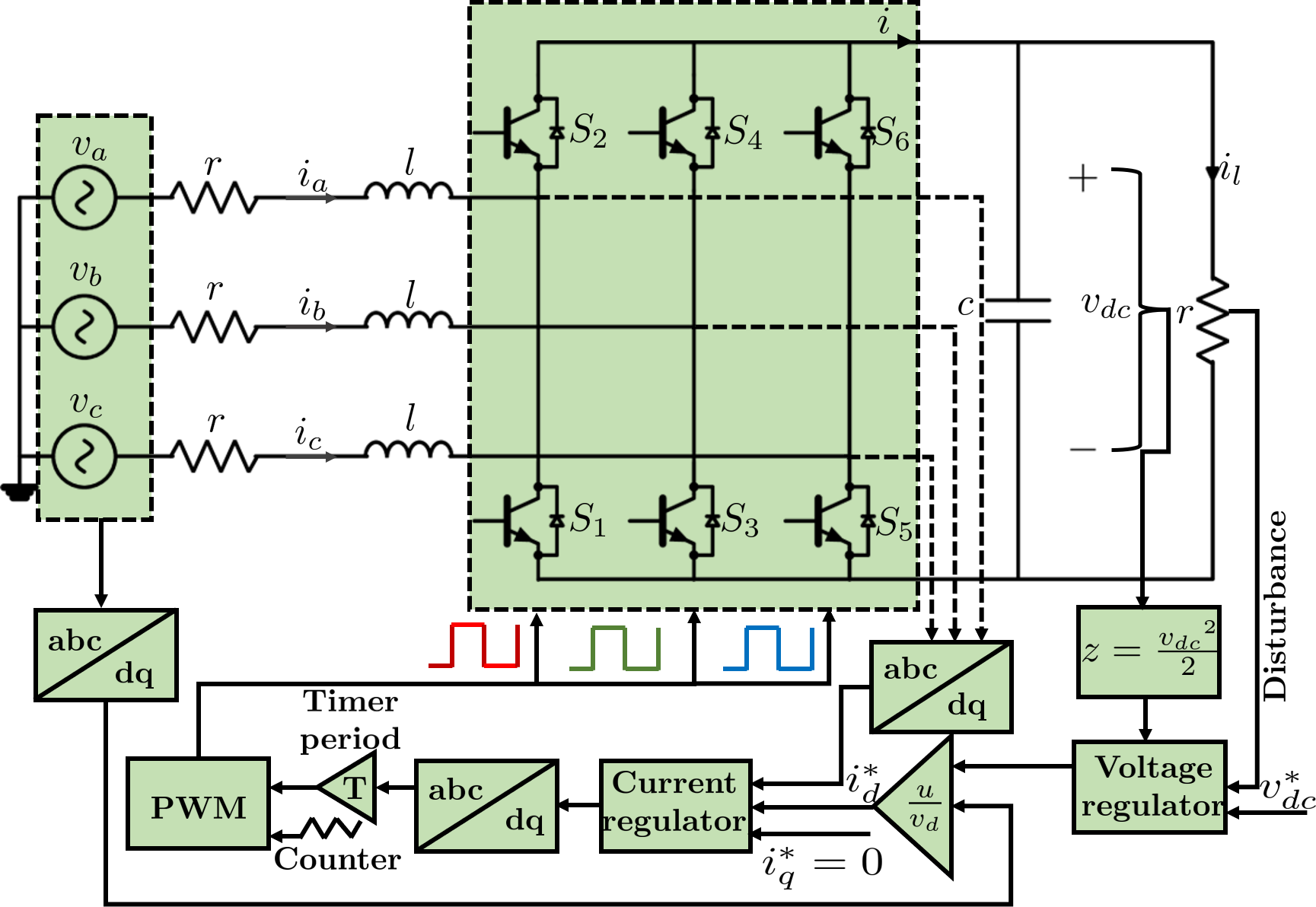}
    \caption{System-level circuit implementation of the proposed adaptive control scheme.}
    \label{AC-DC}
\end{figure}

\subsection{Voltage Regulation}\label{voltage_regulator}
The voltage dynamics from \eqref{power_converter}  can be expressed as:
\begin{equation}\label{voltage-dynamic}
c v_{dc} \frac{dv_{dc}}{dt} = u p_v - \rho(t),
\end{equation}
where  $\rho(t) =  {i}^2_lr$ accounts for the varying load connected to the output DC capacitor of the converter,  $p_v = v_di_d^* + v_qi_q^*$ is the known instantaneous power obtained from the $dq$ current references $(i_d^*, i_q^*)$ and input voltage $(v_d, v_q)$. The control input, $u \in (0, 1)$ is the focus of our design. 

To simplify the analysis, we introduce a change of variable $z = \frac{{v_{dc}}^2}{2}$, where $z$ represents the transformed dynamic variable. Substituting the derivative of $z$ into equation (\ref{voltage-dynamic}), the transformed expression is obtained as:
\begin{align} 
    c\dot{z}= u p_v - \rho \label{z_dot}.
\end{align}
Define the regulation error term as $\tilde{z}_2 = z^* - z $, where  $z^*$ is the transformed voltage reference, and let $\dot{\tilde{z}}_1 = \tilde{z}_2$. Consequently, the augmented error-state dynamics of (\ref{z_dot}) can be expressed as:   
\begin{align} 
\begin{aligned}
\begin{split}
    \dot{\tilde{z}}_1 = {}& \tilde{z}_2, \\
\end{split}\\
\begin{split}
     c\dot{\tilde{z}}_2 = {} & -up_v + \rho.\label{voltage_dyn}\\ 
\end{split}
\end{aligned} 
\end{align} 
The candidate sliding manifold is then proposed as: 
\begin{align}
s = k_1 \tilde{z}_2^{\frac{p}{q}} + \tilde{z}_1 = 0,
\label{sliding_surface}
\end{align}
where $p$ and $q$ are positive odd integers with $\frac{p}{q} \in (1,2)$, and $p > q$, $k_1 > 0$. Taking the first time derivative of both sides of the equation in (\ref{sliding_surface}) yields:
\begin{equation}\label{ss_derivative}
    \dot{s} =  k_1 \frac{p}{q}
    \tilde{z}_2^{\frac{p}{q}-1} \dot{\tilde{z}}_2 + \dot{\tilde{z}}_1.
\end{equation} 
It is sufficient to state that the condition guaranteeing the reachability of \eqref{sliding_surface} is \( s\dot{s} < -\eta |s| \), where \( \eta > 0 \). From \eqref{sliding_surface}, the system states reach the sliding mode, \( s = 0 \), at the time \( t_s = t_0 + T \), where:
\begin{equation}
T = k_1^{\frac{q}{p}} \int_{\tilde{z}_1(t_0)}^{0} \frac{d\tilde{z}_1}{|\tilde{z}_1|^{\frac{q}{p}}} = k_1^{\frac{q}{p}} \frac{p}{p-q} |\tilde{z}_1(t_0)|^{1-\frac{q}{p}}.
\label{T_equation}
\end{equation}
Here, \( \frac{q}{p} \in \left(\frac{1}{2}, 1\right) \). The open-loop dynamics can be obtained by substituting \eqref{voltage_dyn} into \eqref{ss_derivative}, which yields:
\begin{equation}
\dot{s} = \frac{k_1 p}{c q} \tilde{z}_2^{\frac{p}{q}-1} (-u p_v + \rho) + \tilde{z}_2.
\label{open_loop_dynamics}
\end{equation}

Now assuming $\dot{z}$ is not measurable then an estimate, $y$ can be obtained using the derivative filter: 
\begin{align}
 \dot{\eta}_f & = -\sigma \eta_f - \sigma^2 z \label{d1}, \\
 y & = \eta_f + \sigma z \label{d2},
\end{align}
where $\eta_f$ is the derivative filter state and $\sigma >> 0$. The following result provides the bound on the derivative filter error, $e(t) = \dot{z}(t) - \hat{\dot{z}}(t)$. \\

\begin{lemma} \textbf{1:} \label{filter_lemma}
Let $z: \mathbb{R}_+ \rightarrow \mathbb{R}$ be a function satisfying $\underset{t}{\sup}|\dot{z}(t)| \leq \gamma  $, $\underset{t}{\sup} |\ddot{z}(t)|\le \epsilon$. Then for the filter dynamics in \eqref{d1} and \eqref{d2}, there exists, $T < \infty$ such that the  
\begin{equation}
    |\dot{z}(t) - y(t)| \leq \frac{\epsilon}{\sigma},\hspace{2mm} \text{ for all }t > T.
\end{equation}
\end{lemma}

\begin{proof}
See the Appendix.
\end{proof}

As a result, an estimate of the error $\tilde{\rho} = \rho - \hat{\rho}$
where $\hat{\rho}$ represents the disturbance estimate, can be derived from equation \eqref{z_dot} as follows: 
\begin{align}
    \tilde{\rho} = cy - up_v +\hat{\rho}.
\end{align}
Consequently, we consider the the control law:
\begin{align}\label{rho_hat_dot}
\dot{\hat{\rho}} &= \gamma \frac{ k_1 p}{c q}s
    \tilde{z}_2^{\frac{p}{q}-1} + (\epsilon + 1) \textsf{sgn} (\tilde{\rho}),\\\label{control_law_voltage}
u &= \frac{cq }{k_1 p p_v} ( \tilde{z}_2^{2 - \frac{p}{q}} + (\delta + \eta ) \textsf{sgn}(s)) + \hat{\rho} + u_{in},
\end{align} 
where the injected control $u_{in} = ( \tilde{z}_2^{\frac{p}{q}-1} - 1)(\delta + \eta )\textsf{sgn}(s)$ and $\gamma >0$ is the corresponding disturbance estimation adaptation gain and $k_1>0$, $\eta >0$ are control gains.
Substituting \eqref{control_law_voltage} in \eqref{open_loop_dynamics} yields the following closed-loop dynamics:
\begin{align}\label{closed_loop_error_dynamics}
\dot{\hat{\rho}}  & =\gamma \frac{ k_1 p}{c q}s
    \tilde{z}_2^{\frac{p}{q}-1} + (\epsilon + 1) \textsf{sgn} (\tilde{\rho}),\\ 
 \dot{s} & =   - (\delta + \eta ) \textsf{sgn}(s) \tilde{z}_2^{\frac{p}{q}-1} + \frac{ k_1 p}{c q}
    \tilde{z}_2^{\frac{p}{q}-1} (\rho -   \hat{\rho}) + u_{in}. \label{closed_loop_error_dynamics2}
    \end{align} 
The next results give the finite-time stability of the closed-loop dynamics above.
\begin{theorem}
Consider the closed-loop dynamics described in \eqref{closed_loop_error_dynamics} and \eqref{closed_loop_error_dynamics2}. If $\rho(t)$ is uniformly continuous, and there exists constants $\eta > 0, \epsilon > 0$ and $\gamma > 0$,
     then the sliding manifold defined by \eqref{sliding_surface} is reached within the finite time interval 
     $[0,\hspace{2mm} t_r]$,  and $s(t) \equiv 0$ for all $t> t_r$ where the finite reaching time $t_r$ is given by:
     \begin{align}
         t_r = \frac{{\max\left\{\gamma,1\right\}}}{\sqrt{2}\min\left\{\gamma\left(\delta + \eta\right), 1 \right\}  \min \{\sqrt{\gamma}, 1 \}}   \norm{
\begin{bmatrix}
           |s(0)|  \\
           |\tilde{\rho} (0)|
         \end{bmatrix}}_2 .
           \end{align}
\end{theorem}
\begin{remark}
Moroever, the disturbance estimate error, $\tilde{\rho}(t)$ converges to zero in finite time.
\end{remark}

\begin{proof}
    Consider the following Lyapunov candidate:
    \begin{align}\label{LCF}
    V = \frac{1}{2}s^2 + \frac{1}{2\gamma}\tilde\rho^2,
    \end{align}
    where $\tilde\rho = \rho-\hat{\rho}$.
    
Taking the first time-derivative of (\ref{LCF}) and substituting (\ref{closed_loop_error_dynamics2}) yields:
    \begin{alignat}{1}
\dot{V} & =  s\dot{s} +\frac{1}{\gamma} \tilde{\rho} \dot{\tilde{\rho}}, \nonumber \\
        & =    - (\delta + \eta ) \lvert s \rvert \tilde{z}_2^{\frac{p}{q}-1} + \frac{ k_1 p}{c q}s
    \tilde{z}_2^{\frac{p}{q}-1} \tilde{\rho} +\frac{1}{\gamma} \tilde{\rho} \dot{\tilde{\rho}}  \\
    & + ( \tilde{z}_2^{\frac{p}{q}-1} - 1)(\delta + \eta )\lvert s \rvert, \nonumber \\
   & 
  = - (\delta + \eta ) |s| + \frac{ k_1 p}{c q}s
    \tilde{z}_2^{\frac{p}{q}-1} \tilde{\rho} + \frac{1}{\gamma} \tilde{\rho} \dot{\rho} -  \frac{1}{\gamma} \tilde{\rho} \dot{\hat{\rho}}.\nonumber  
\end{alignat} 
Consequently, from Assumption \ref{disturbance_assumption}, we have:
  \begin{alignat}{1}
\dot{V} & \leq  - (\delta + \eta ) |s| + \frac{ k_1 p}{c q}s
    \tilde{z}_2^{\frac{p}{q}-1} \tilde{\rho} + \frac{\epsilon}{\gamma} |\tilde{\rho}|  -  \frac{1}{\gamma} \tilde{\rho} \dot{\hat{\rho}}.\nonumber  
\end{alignat}
Then applying the adaptation law in (\ref{rho_hat_dot}) yields: 
\begin{align}
    \dot{V} & \leq  -(\delta + \eta ) |s| - \frac{1}{\gamma}|\tilde{\rho}|, \nonumber \\
    & \leq -\min \left\{\delta + \eta, \frac{1}{\gamma}   \right\}
   \norm{
\begin{bmatrix}
           |s|  \\
           |\tilde{\rho}|
         \end{bmatrix}}_1 . \nonumber \label{V_dot_bound}
\end{align}
Given that $\dot{V} < 0$ and $V > 0$, $V$ is bounded, which implies $V \in \mathcal{L}_\infty$. By implication, $s \in \mathcal{L}_\infty$, $\tilde{\rho} \in \mathcal{L}_\infty$. Thus, $s$ and $\tilde{\rho}$ are bounded.  

From (\ref{LCF}), it is clear that: 
\begin{align*}
V&\le\frac{1}{2\min\left\{\gamma,1\right\}}\norm{\begin{bmatrix}|s| \\|\tilde{\rho}|\end{bmatrix}}_2^2
 \le\frac{1}{2\min\left\{\gamma,1\right\}}\norm{\begin{bmatrix}|s| \\|\tilde{\rho}|\end{bmatrix}}_1^2.
\end{align*}

Thus,
\begin{align}\nonumber
    \dot{V} &\leq -2\min\left\{\delta + \eta, \frac{1}{\gamma}\right\}\min\left\{\gamma,1\right\}V^\frac{1}{2},\\
            & \leq -2\frac{\min\left\{\gamma\left(\delta + \eta\right), 1 \right\}}{\max\left\{\gamma,1\right\}}V^\frac{1}{2}. \nonumber
\end{align}
Now let $\psi(t)\in\mathbb{R}_+$ be the solution of the system:$
     \frac{d}{dt}\psi(t)  = -k \psi(t)^{\frac{1}{2}},
$
where $k = 2\frac{\min\left\{\gamma\left(\delta + \eta\right), 1 \right\}} {\max\left\{\gamma,1\right\}}$. It follows, using the comparison lemma \cite{khalil2002nonlinear}, $
    V(t) \leq \psi(t) = \psi(0) - \frac{k}{2}t$. Consequently, if $|V(0)| \neq 0$, the inequality implies:
\begin{equation} \label{Vfinite_time}
    V(t) \equiv 0, \hspace{2mm} \forall \hspace{2mm} t > \frac{{\max\left\{\gamma,1\right\}}}{\min\left\{\gamma\left(\delta + \eta\right), 1 \right\}}|V^{\frac{1}{2}}(0)|,
\end{equation}
Furthermore, if $s(0) \neq 0$, $\tilde{\rho}(0) \neq 0$, $V(t) \equiv 0$ for all:

\begin{align}
   \hspace{1mm} t > \frac{{\max\left\{\gamma,1\right\}}}{\sqrt{2}\min\left\{\gamma\left(\delta + \eta\right), 1 \right\}  \min \{\sqrt{\gamma}, 1 \}}   \norm{
\begin{bmatrix}
           |s(0)|  \\
           |\tilde{\rho} (0)|
         \end{bmatrix}}_2 .
\end{align}
Therefore, $V$ converges to the origin in a finite time, as specified by (\ref{Vfinite_time}). As a result, the system states converge to the origin within a finite time as $s(t)$ and $\tilde{\rho}(t)$ approaches zero. 
\end{proof}
\subsection{Current Regulation}\label{CurrentRegulation}
In this section the regulation of the converter grid currents to the references $i_d^* = \frac{u}{v_d}$, $i_q^*=0$  is carried out where $u$ is the voltage control law described in section \ref{voltage_regulator}. From (\ref{power_converter}), the current dynamics of the power converter is given by:
\begin{equation} \label{power_conv}
l\dot{\mathbf{i}} = (-r  + w_g l J) \mathbf{i} + \mathbf{v} - \mathbf{u} v_{dc} + \boldsymbol{\rho},
\end{equation}
where the external grid disturbance is represented by $\boldsymbol{\rho} = [\rho_d \hspace{5pt}\rho_q]^\top$, and it is constrained such that $||\boldsymbol{\rho}|| \leq \delta$, with $\delta > 0$.

For a given reference trajectory $\mathbf{i}_0 = [i_d^* \ i_q^*]^\top$, the error dynamics of the system in \eqref{power_conv} can be derived by taking the time derivative of the error signal $\mathbf{\tilde{i}}(t) = \mathbf{i}(t)-\mathbf{i}_0(t)$ and substituting (\ref{power_conv}). This results in the following expression:
\begin{equation}\label{error_dynamics2}
 l \mathbf{\dot{\tilde{i}}} = -r(\mathbf{\tilde{i}} + \mathbf{i}_0) + w_g l J (\mathbf{\tilde{i}} + \mathbf{i}_0) + \mathbf{v} - \textbf{u} v_{dc} + \boldsymbol{\rho} - l \mathbf{\dot{i}_0}. 
\end{equation}
Let the known term, $\boldsymbol{\psi} = -r \mathbf{i}_0 +  w_g l J \mathbf{i}_0 + \mathbf{v} -  l \mathbf{\dot{i}_0} $
so that the error dynamics in \eqref{error_dynamics2} can be re-expressed as:
\begin{equation} \label{error_dynamics}
l \mathbf{\dot{\tilde{i}}} = -r \mathbf{\tilde{i}} +   w_g l J \mathbf{i} - \mathbf{u} v_{dc} + \boldsymbol{\rho} + \boldsymbol{\psi}.
\end{equation}
For the system described by (\ref{error_dynamics}), we consider a sliding manifold defined by the nonlinear equation:
\begin{equation}\label{sm_current}
    \mathbf{s}(i_d,i_q) = \mathbf{\tilde{i}} + \beta \int_0^t \mathbf{\tilde{i}}^ \frac{q}{p} (\tau) d \tau = \mathbf{0},
\end{equation} 
where $\beta > 0$, $p$ and $q$ are positive odd integers with $p > q$ and $\frac{p}{q} \in (1, 2)$. On the sliding manifold ($\mathbf{s} = \mathbf{0}$), the converter current states are described by the dynamics:
\begin{equation} \label{i_tilde_dot}
\mathbf{\dot{\tilde{i}}} = -\beta  \mathbf{\tilde{i}}^ \frac{q}{p}.
\end{equation}
Assuming a sufficient reaching law is applied to (\ref{error_dynamics}) then on the sliding manifold, $\mathbf{\tilde{i}}$ will be drawn to the system's terminal attractor (the origin). That is, as  $\mathbf{\tilde{i}}$ goes to zero by implication of \eqref{i_tilde_dot}, the converter states, $i_d$ and $i_q$ are forced to go to zero by the finite time given by:
\begin{equation}\label{finite_time}
    t_f = -\frac{1}{\beta} \int_{\mathbf{\tilde{i}}(0)}^0 \frac{\text{d} \mathbf{\tilde{i}}}{|\mathbf{\tilde{i}}|^ \frac{q}{p}} = \frac{p}{\beta(p - q)} \mathbf{\tilde{i}}(0)^{1-\frac{q}{p}}, \hspace{0.15in}\mathbf{\tilde{i}}(0) \neq 0.
\end{equation}
Taking the first-time derivative of the sliding surface along the system trajectories in (\ref{sm_current}) yields:
\begin{equation}\label{s_dot}
    \dot{\mathbf{s}} = \mathbf{\dot{\tilde{i}}}  + \beta \mathbf{\tilde{i}}^ \frac{q}{p}.
\end{equation}
Substituting (\ref{error_dynamics}) into (\ref{s_dot}) gives the open loop dynamics as:
\begin{equation} \label{open_loop}
    l\dot{\mathbf{s}} = (-r  +   w_g l J )\mathbf{\tilde{i}} - \mathbf{u} v_{dc} + \boldsymbol{\rho} + \boldsymbol{\psi} + l  \beta \mathbf{\tilde{i}}^ \frac{q}{p}.
\end{equation}
Consequently, the following robust control law is proposed to drive the rectifier states to the sliding surface and guarantee they stay there:
\begin{align} \label{control_law}
   \mathbf{u} = \frac{1}{v_{dc}} \Big((-r  +   w_g l J )\mathbf{\tilde{i}} + \boldsymbol{\psi} + l \beta  \mathbf{\tilde{i}}^ \frac{q}{p}  
   +  (\delta + \eta ) 
    \textsf{sat} (\frac{\mathbf{s}}{\epsilon}) \Big),
\end{align}
where $\eta > 0$, $\epsilon > 0$ are control design parameters. Substituting the control law from \eqref{control_law} into \eqref{open_loop} yields the following closed-loop dynamics:
\begin{equation} \label{closed_loop}
    l \dot{\mathbf{s}} = -(\delta + \eta)  \textsf{sat}(\frac{\mathbf{s}}{\epsilon}) +  \boldsymbol{\rho}. \end{equation}
 
\begin{lemma}\label{lem:sat}
\textbf{2:} For any $\epsilon > 0$ and $\mathbf{s} \in \mathbb{R}^n$ with $\|\mathbf{s}\| \le \epsilon \sqrt{n}$,  
\[
\mathbf{s}^\top \textsf{sat}\Big(\frac{\mathbf{s}}{\epsilon}\Big) \ge \sqrt{n}\,\|\mathbf{s}\|.
\]
\end{lemma}

\begin{proof}
See the Appendix.
\end{proof}

\begin{theorem}
    Consider the closed-loop error dynamics described in \eqref{closed_loop}. If $\left\|\mathbf{s}(0)\right\|\le \epsilon\sqrt{n}$ for some $\epsilon >0$ and $n\ge 1$,
     then $\mathbf{s}(t)\equiv \mathbf{0}$ is reached within the finite-time interval $[0,\hspace{2mm}l\frac{\epsilon}{\eta}]$, and $\mathbf{s}(t) \equiv \mathbf{0}$ for all $t> l\frac{\epsilon}{\eta}$. 
\end{theorem}

\begin{proof}
    Consider the following energy function:
    \begin{equation}\label{V}
       V = \frac{l}{2} \mathbf{s}^\top \mathbf{s},
   \end{equation}
   then taking the first-time derivative of V and substituting  (\ref{closed_loop}) yields:
\begin{align}
   l \mathbf{s}^\top \dot{\mathbf{s}}  = -(\delta + \eta) \epsilon \frac{\mathbf{s}^\top}{\epsilon}\textsf{sat}\left(\frac{\mathbf{s}}{{\epsilon}}\right) + \mathbf{s}^\top \boldsymbol{\rho}.
  \end{align}
Using Lemma 2, it follows that for all $\norm{\mathbf{s}} \leq \epsilon \sqrt{n}$ 
\begin{align*}
l \mathbf{s}^\top \dot{\mathbf{s}}  & \leq -(\delta + \eta)\sqrt{n} \norm{\mathbf{s}} + \mathbf{s}^\top \boldsymbol{\rho}, \\
& \leq -((\delta + \eta)\sqrt{n}  - \delta) \norm{\mathbf{s}}. 
\end{align*}
Since $n\ge 1$,
\begin{equation}\label{V_dot}
   \dot{V}  \leq -\eta\sqrt{n} \norm{\mathbf{s}}. 
   \end{equation}
Thus, the choice of control law satisfies the reachability condition in (\ref{V_dot}). Consequently the reachability time is considered by first expressing (\ref{V_dot}) in the the
form:
\begin{equation}
    \frac{l}{2}\frac{d \norm{\mathbf{s}}^2}{dt} \leq -\eta \sqrt{n} \norm{\mathbf{s}}.
\end{equation}
Now let $\phi(t)\in\mathbb{R}_+$ be the solution of the following differential equation:
\begin{equation}
    \frac{l}{2} \frac{d}{dt}\phi^2(t)  = -\eta \sqrt{n} \phi(t),
\end{equation}
The comparison lemma \cite{khalil2002nonlinear}, guarantees $\norm{\mathbf{s}} \leq \phi(t) = \phi(0) - \frac{\eta \sqrt{n}}{l}t$.
It is obvious then, if $\norm{\mathbf{s}(0)} \neq 0$, the system trajectories will reach the sliding surface within the finite time interval  $\big[0, \frac{l \norm{\mathbf{s}(0)}}{k} \big]$ and remain on $\mathbf{s} = \mathbf{0}$ for all  $t > \frac{l \norm{\mathbf{s}(0)}}{k}$.
This completes the proof.
\end{proof}

\section{Numerical and Experimental Results}
This section describes the performance of the proposed rectifier control laws on simulation, and on a physcial rectifier circuit.  To test how well the controllers perform, comparisons with recent algorithms developed for rectifier regulation, such as Proportional-Integral and Cascaded Proportional-Resonant Control Method, \cite{guler2023simplified}, adaptive STA, \cite{luo2019adaptive} and finite time integral terminal SMC \cite{eskandari2020finite} are presented.  

\subsection{Numerical Simulation Results}
To validate the performance of the proposed control design against the benchmarks earlier mentioned, simulations involving voltage and current regulations are carried out with the rectifier parameters shown in Table \ref{Table1} and control parameters given in Table \ref{table:table3}.    
\begin{table}[h]
\begin{center}
\caption{Rectifier Model Parameters.}
\label{Table1}
\begin{tabular}{ll}
\hline \hline
Parameter & Value \\
\hline
Line to line voltage, $\mathbf{v}$ & $\SI{400}{\volt}_{rms}$ (L-L)\\
Approx.\ DC-link, $v_{dc,avg}$ & $\approx$ 540 V \\
Grid resistance, $r$ & 0.02 $\Omega $ \\
Capacitance, $c$ & $3300 \hspace{2pt} \mu \text{F}$\\
Inductance, $l$  & 0.5$ \hspace{2pt} \text{mH}$ \\
Grid frequency, $f$ & 60 Hz \\
Reference voltage, $v^*_{dc} $ & 520V \\
\hline \hline
\end{tabular}
\end{center}
\end{table}

\begin{table}[h]
    \centering
    \caption{Control parameters for: Method 1 – Proportional-Resonant Control; Method 2 – Adaptive STA; Method 3 – Finite-Time Integral Terminal SMC; and the proposed method.}
    \label{table:table3}

    \begin{tblr}{
      colspec={Q[2cm, c] Q[2cm, c] Q[2cm, c] Q[2.8cm, c] Q[2.8cm, c]},
      hlines={1pt},
      row{1} = {font=\bfseries, valign=m},
      cells={valign=m}
    }
    \hline
    Loop &
    \makecell{Method 1\\$(k_p, k_i)$} &
    \makecell{Method 2\\$(\alpha_1, \alpha_2)$} &
    \makecell{Method 3\\$(\zeta, \mu, \sigma,$$p, q, p_1, q_1)$} &
    \makecell{Proposed Method\\$(p, q, \gamma, k_1,$$\delta, \eta, \beta, \epsilon)$} \\
    \hline
    Current &
    \makecell{$(0.01, 1)$} &
    \makecell{$(30, 2)$} &
    \makecell{$(1, 1, 0.5,$\\$5, 3, 1, 2)$} &
    \makecell{$(5, 3, -,$\\$-, 2, 1, 0.5, 1)$} \\
    Voltage &
    \makecell{$(0.01, 1)$} &
    \makecell{$(20, 4)$} &
    \makecell{$(1, 1, 0.5,$\\$5, 3, - , -)$} &
    \makecell{$(5, 3, 0.5,$\\$1, 2, 1, - , 1)$} \\
    \hline
    \end{tblr}
\end{table}

The response of the proposed voltage controller on the output DC voltage of the rectifier under a sinusoidal load disturbance is shown in Fig. \ref{Voltage_reg}. Figs. \ref{Voltage_reg}(a)–(d) present the responses of the proposed method and the rectifier control methods from \cite{guler2023simplified}, \cite{luo2019adaptive}, and \cite{eskandari2020finite}. A fast convergence of the proposed adaptive-based controller is observed, achieving the desired set point within $t_c = 0.8$ ms. This represents an improvement in convergence time, $t_c$, of $93.50\%$, $99.40\%$, and $74.19\%$ compared to the methods in \cite{guler2023simplified}, \cite{luo2019adaptive}, and \cite{eskandari2020finite}, respectively. The control gain, $\eta$, can be increased during simulations for even faster convergence. Fig. \ref{Voltage_reg}(e) and Fig. \ref{Voltage_reg}(f) show the load disturbance online estimation performance using an extended state observer (ESO) \cite{feng2002non} and the proposed disturbance adaptation law, respectively. As observed, our method accurately estimates the load power with smaller maximum and minimum overshoots. By adjusting the update parameter ($\gamma > 0$), the initial transient spike can be further reduced. The comparison of methods is considered fair due to the closely matched accumulated energies of the controllers, as shown in Fig. \ref{CLA}.

Fig. \ref{fig:example} shows that the current on the d-axis tracks a bounded reference derived from the voltage loop. Similarly to the voltage case, Fig. \ref{Current_Comparision_baselines} presents the current response for the baseline controllers under constant $dq$ reference values of \SI{1}{\ampere} and \SI{0}{\ampere}. The proposed method is seen to drive the $dq$ currents to the reference values with no overshoot, demonstrating the practical viability of the control scheme. It can also be seen that the proposed method converges faster than the other methods. The convergence times are $t_1 = 10$ ms, $t
_2 = 50$ ms, $t_3 = 80$ ms and $t_4 = 30$ ms for the proposed current controller, method from \cite{guler2023simplified}, \cite{eskandari2020finite} and \cite{luo2019adaptive} respectively. Based on the convergence times, the proposed method is seen to have $80 \%$, $87.5 \%$, and $66.7 \%$ improvement over the baseline controllers respectively. Futhermore, the controller remains robust during a 60 to 59 Hz grid-frequency drop over 2 s, maintaining accurate, undistorted AC-current tracking as shown in Fig.~\ref{frequency_drop}. Similarly, Fig. \ref{grid_resistance_var} shows the controller's resilience to parameter variations when the grid inductance and resistance vary by 10\%.

\begin{figure}[!htpb]
    \centering \includegraphics[scale=0.6]{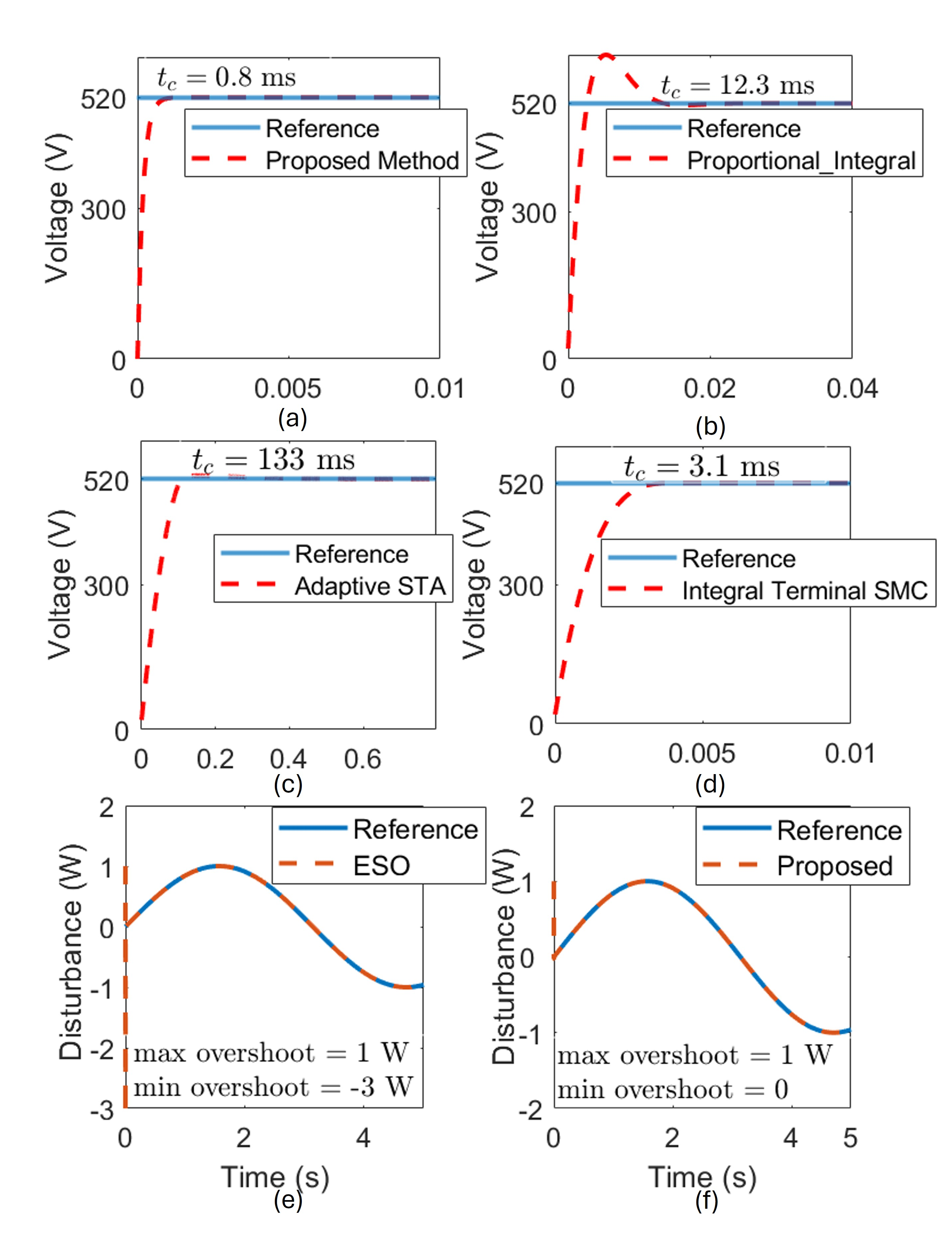}
\caption{DC-link voltage responses under varying load disturbances using (a) the proposed method, (b) the proportional–integral cascaded method, (c) the adaptive STA method, (d) the integral terminal SMC, (e) disturbance estimation using the ESO and (f) disturbance estimation using the proposed method.}
    \label{Voltage_reg}
\end{figure}

\begin{figure}[!htpb]
    \centering \includegraphics[scale=0.8]{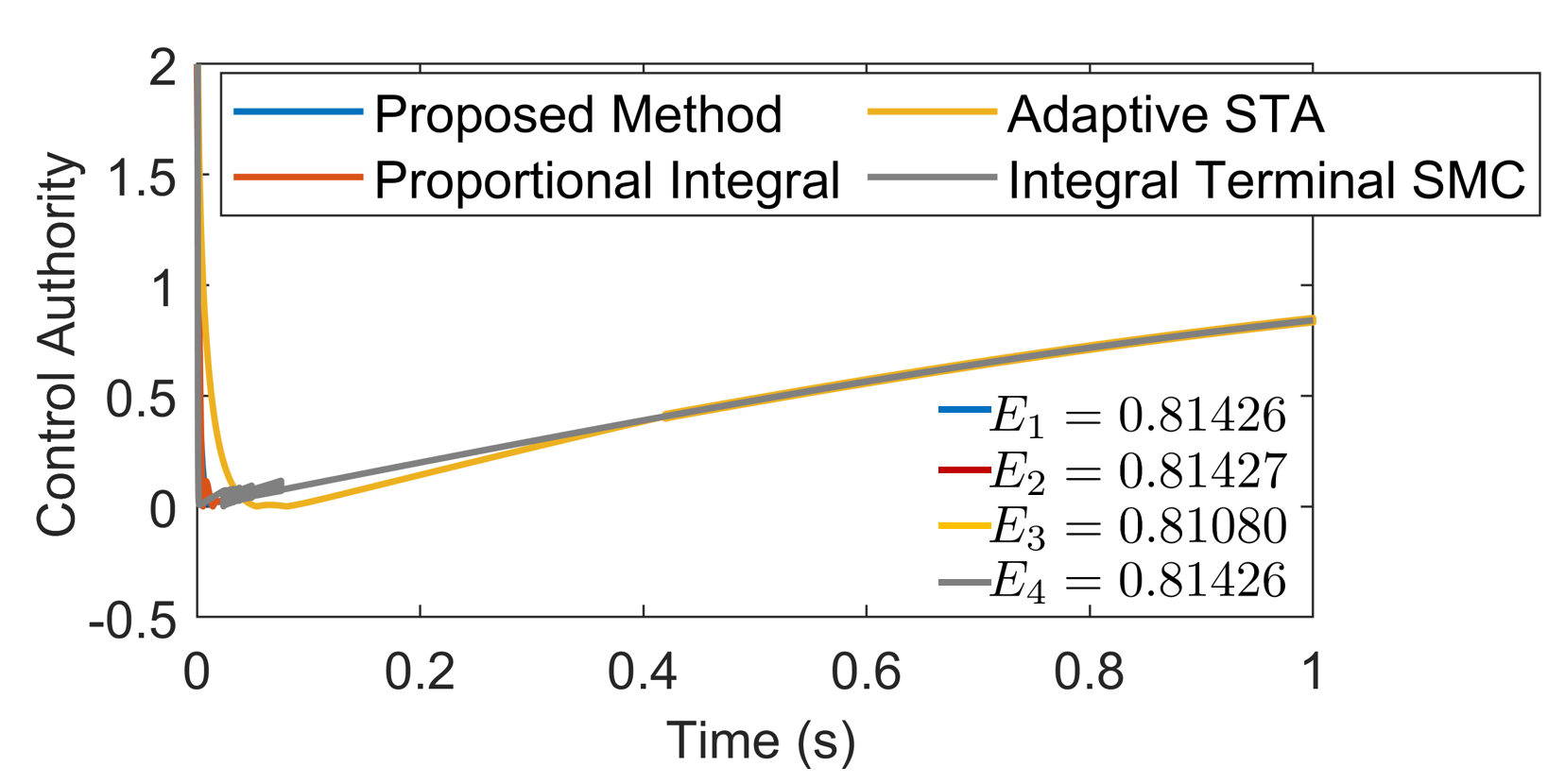}
    \caption{Comparison of the voltage control signals for the proposed method with existing methods. $E_1$, $E_2$, $E_3$ and $E_4$ are the respective control signal norms computed.}
    \label{CLA}
\end{figure}

\begin{figure}[!htpb]
    \centering \includegraphics[scale=0.9]{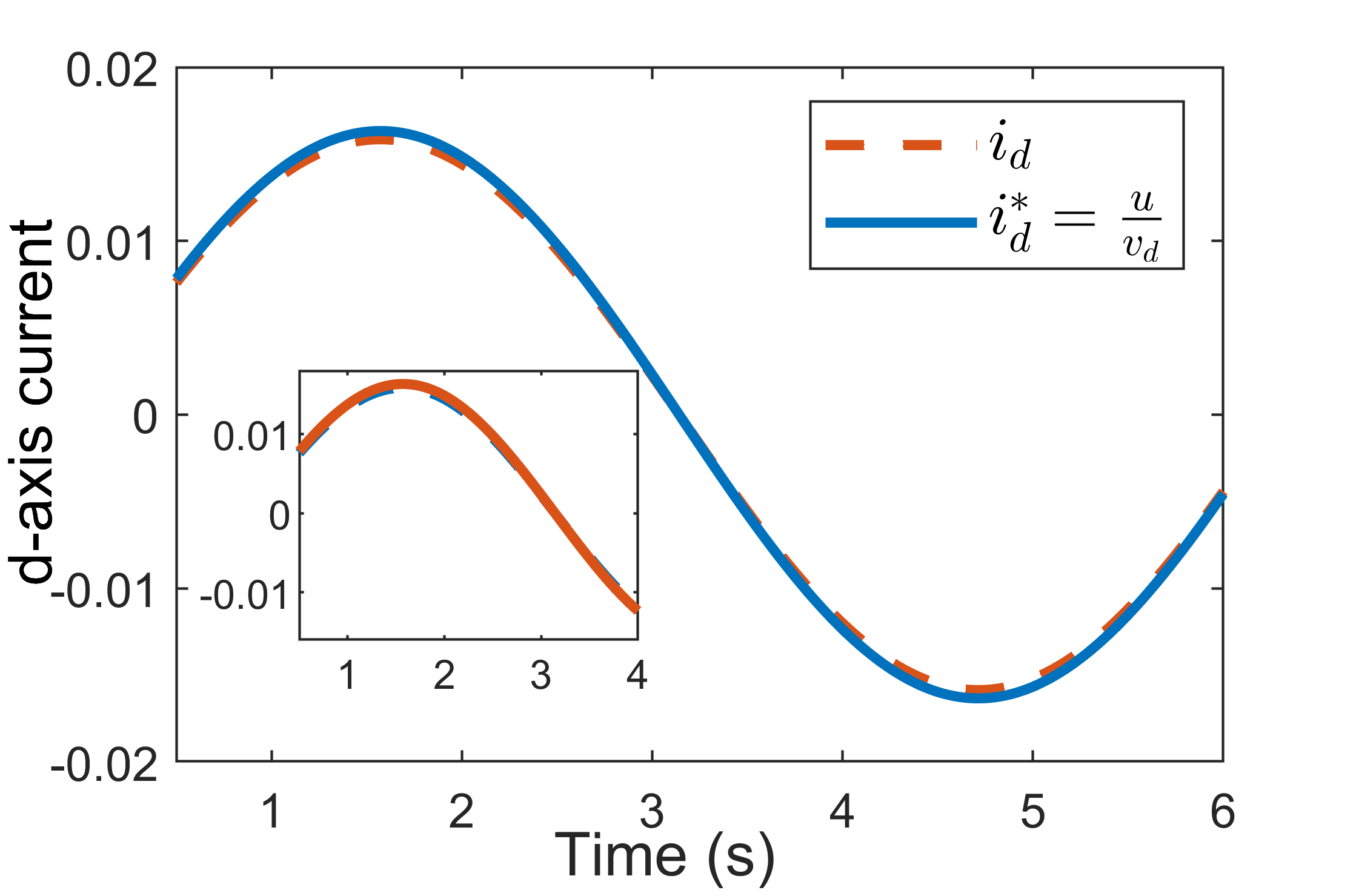}
    \caption{$d$-axis current response demonstrating effective tracking of the reference signal derived from the voltage regulation loop.}
    \label{fig:example}
\end{figure}

\begin{figure}[!htpb]
\centering
\includegraphics[scale=0.4]{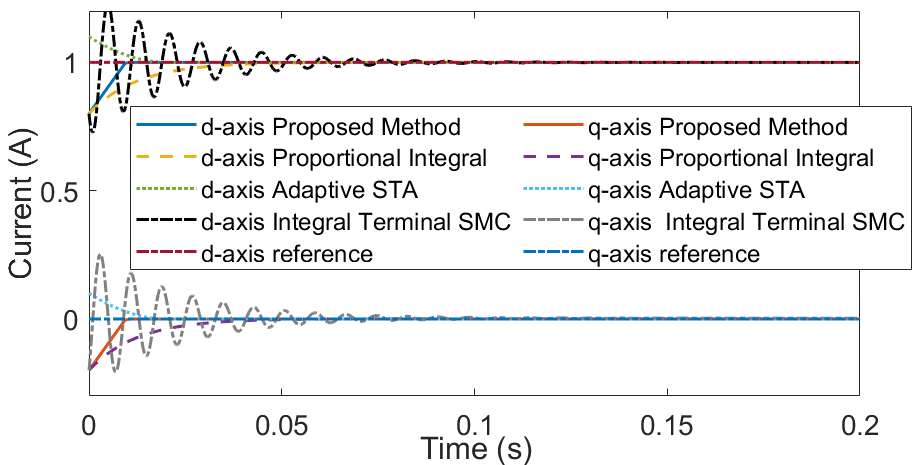} 
\caption{Current regulation comparison across the proposed method, the proportional–integral cascaded method, the adaptive STA method, and the integral terminal SMC.}  
\label{Current_Comparision_baselines}
\end{figure} 

 \begin{figure}[!htpb]
 \centering
\includegraphics[scale=0.5]{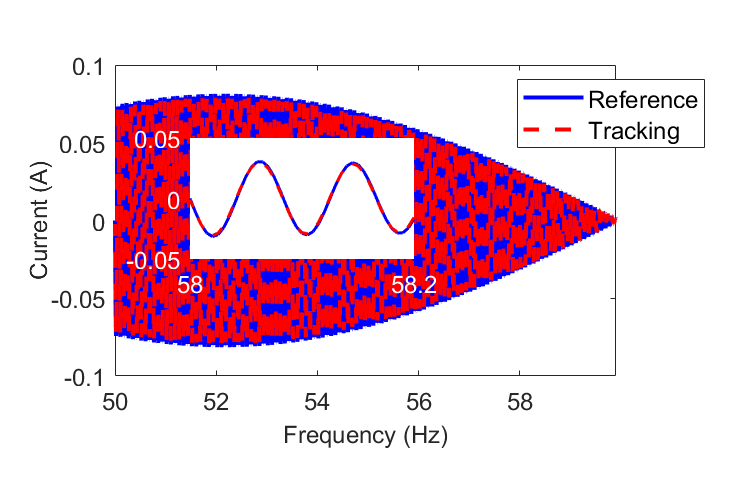} 
\caption{Response of Phase A current under a slow grid frequency drop - 60 to 59 Hz drop over 2 s.}
\label{frequency_drop}
\end{figure} 

\begin{figure}[!htpb]
 \centering
\includegraphics[scale=0.65]{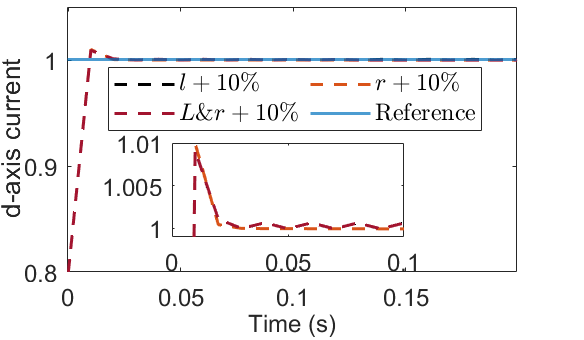} 
\caption{$d$-axis current response of the proposed controller considering $\pm10\%$ variations in grid inductance ($l$) and resistance ($r$). The results demonstrate accurate current tracking and robustness of the control law against parameter uncertainties.}
\label{grid_resistance_var}
\end{figure}

\subsection{Experimental Results}
The hardware validation of the proposed controllers was conducted using the SYNDEM reconfigurable research converters shown in Fig.~\ref{Setup}. The setup includes two SYNDEM converters, a DC source, host PC, power resistors, oscilloscope, RS-485 link, JTAG emulator, and all necessary cabling. Each converter contains a control board featuring a TI C2000 TMS320F28335 control card with signal conditioning circuitry and a power board with IGBT modules, passive components, sensors, relays, and diode bridges. The controllers were implemented on the host PC and deployed via JTAG, with RS-485 enabling serial communication. For safety and equipment protection, all hardware experiments were carried out at reduced laboratory voltage levels, which differ from the higher voltages used in simulation. In this configuration, one converter operated as an inverter and the other as a rectifier: the inverter converted a \SI{60}{\volt} DC input into a three-phase \SI{35}{\volt_{rms}} (L-L) AC output using bipolar switching values appropriate for low-voltage laboratory testing and this AC output then served as the input to the rectifier. After powering the system, waveforms were monitored on an oscilloscope, and various test scenarios were conducted to evaluate controller performance.
\begin{figure}[h]
    \centering \includegraphics[width = 0.7\textwidth]{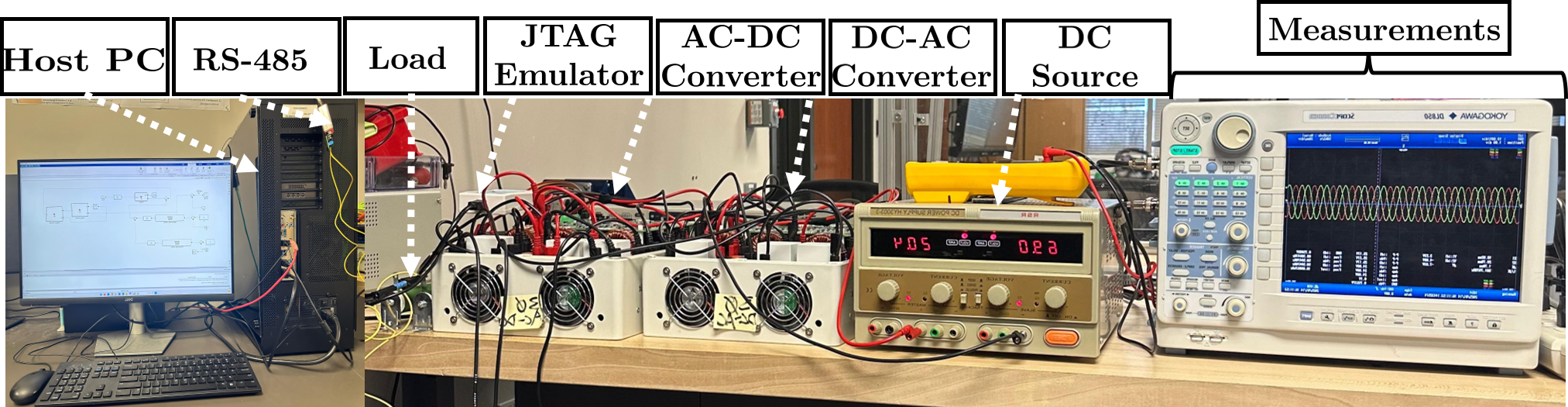}
    \caption{Laboratory experimental setup demonstrating the controller integration with the hardware rectifier circuit.}
    \label{Setup}
\end{figure}
\subsubsection{Scenario 1 - Step Voltage Reference} A laboratory step test is performed by increasing the DC voltage reference from \SI{30}{\volt} to \SI{40}{\volt} at \( t = 1 \) s, with the same procedure applied to the controller in \cite{guler2023simplified}. The resulting voltage responses are shown in Figs.~\ref{proposed_volt_reg} and \ref{pi_based_reg}. Both controllers track the step, but the PI Cascaded based method exhibits a rise time of \( t_r = 1.6 \) s, while the proposed regulator improves this by \( 0.2 \) s. Ripple magnitudes (pk–pk) for \( t > 1 \) s, along with rise time and steady-state error, are summarized in Table~\ref{Responsechar}, demonstrating that the proposed controller achieves lower ripple and smaller steady-state error than the baseline.

\begin{figure}[!htpb]
    \centering
    \includegraphics[scale=0.68]{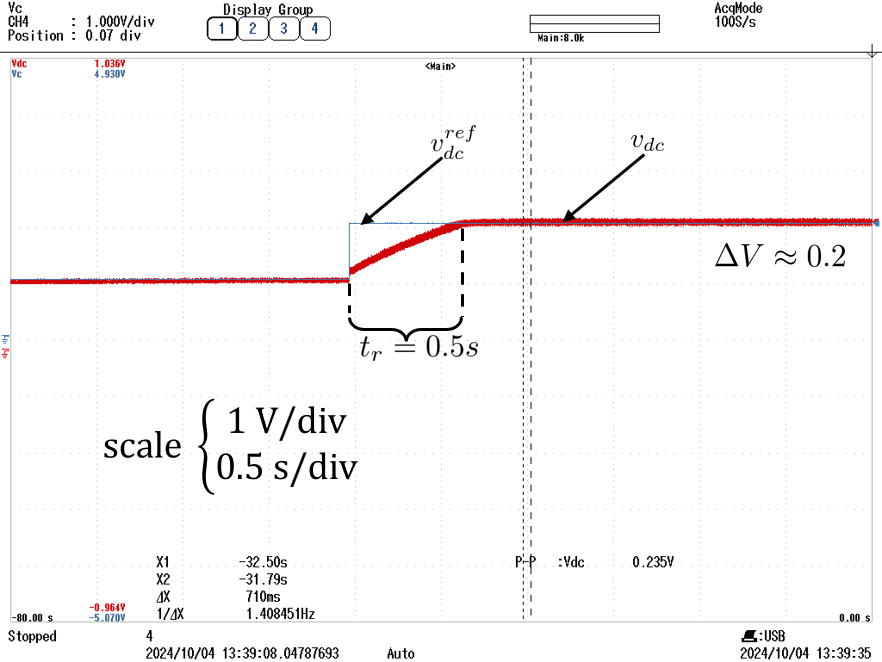}
       \caption{Experimental response of the proposed voltage regulator to a step change in the DC reference voltage.}
    \label{proposed_volt_reg}
\end{figure}
\begin{figure}[!htpb]
    \centering
\includegraphics[scale=0.63]{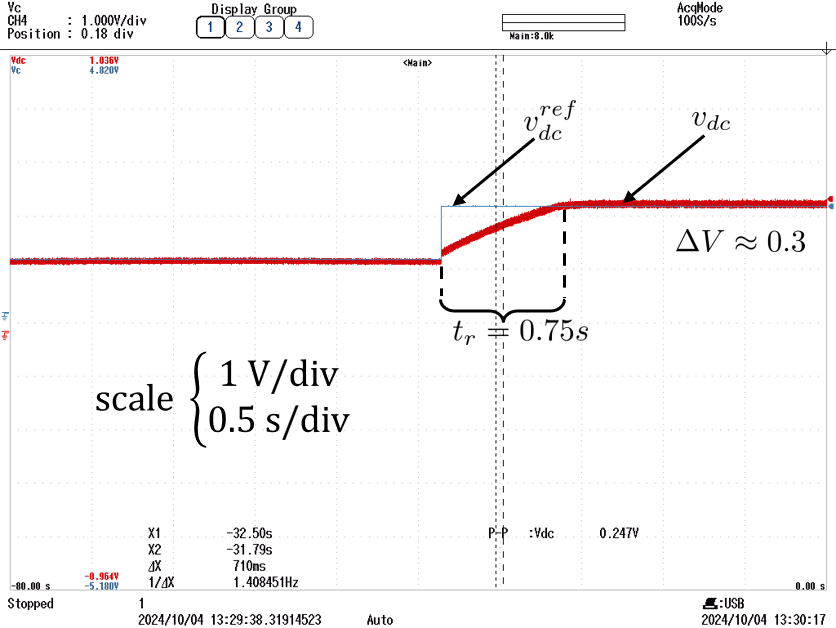}
       \caption{Experimental response of the method from \cite{guler2023simplified} to a step change in the DC reference voltage.}
    \label{pi_based_reg}
\end{figure}

\begin{table}[!htpb]
    \begin{center}
\caption{Experimental DC Voltage Response Performance.}
        \begin{tblr}{
          colspec={Q[4cm] | Q[2.8cm] | Q[2.8cm] | Q[2.8cm]},
          columns={halign=c},hlines
        }
        \hline
          Controller &  Ripple Magnitude ($\Delta V$) & Rise Time, $t_r$ (s) & Steady-State Error\\
Method from \cite{guler2023simplified} & 0.30 & 0.75 & 0.37\\
Proposed & 0.20 & 0.50 &0.25\\
Improvement($\%$) & 33.33 & 33.33 & 32.43\\
                 \hline      
        \end{tblr}
        \label{Responsechar}
    \end{center}
\end{table}

\subsubsection{Scenario 2 - Current Tracking under Constant Load} Next we present the rectifier grid currents output response under a constant $\SI{12}{\ohm}$ load with references obtained from the voltage control loop as done in simulation. The normalized response for phase A currents is shown in Fig. \ref{current_track_experiment} respectively as well as the chattering amplitudes in Fig. \ref{phase_a_chattering_resp}. As seen, the proposed current regulator was able to track the reference trajectory even as the reference changed in magnitude with low chattering amplitudes $(< 6\times 10^{-3})$.
\begin{figure}[!htpb]
    \centering
    \includegraphics[scale=0.25]{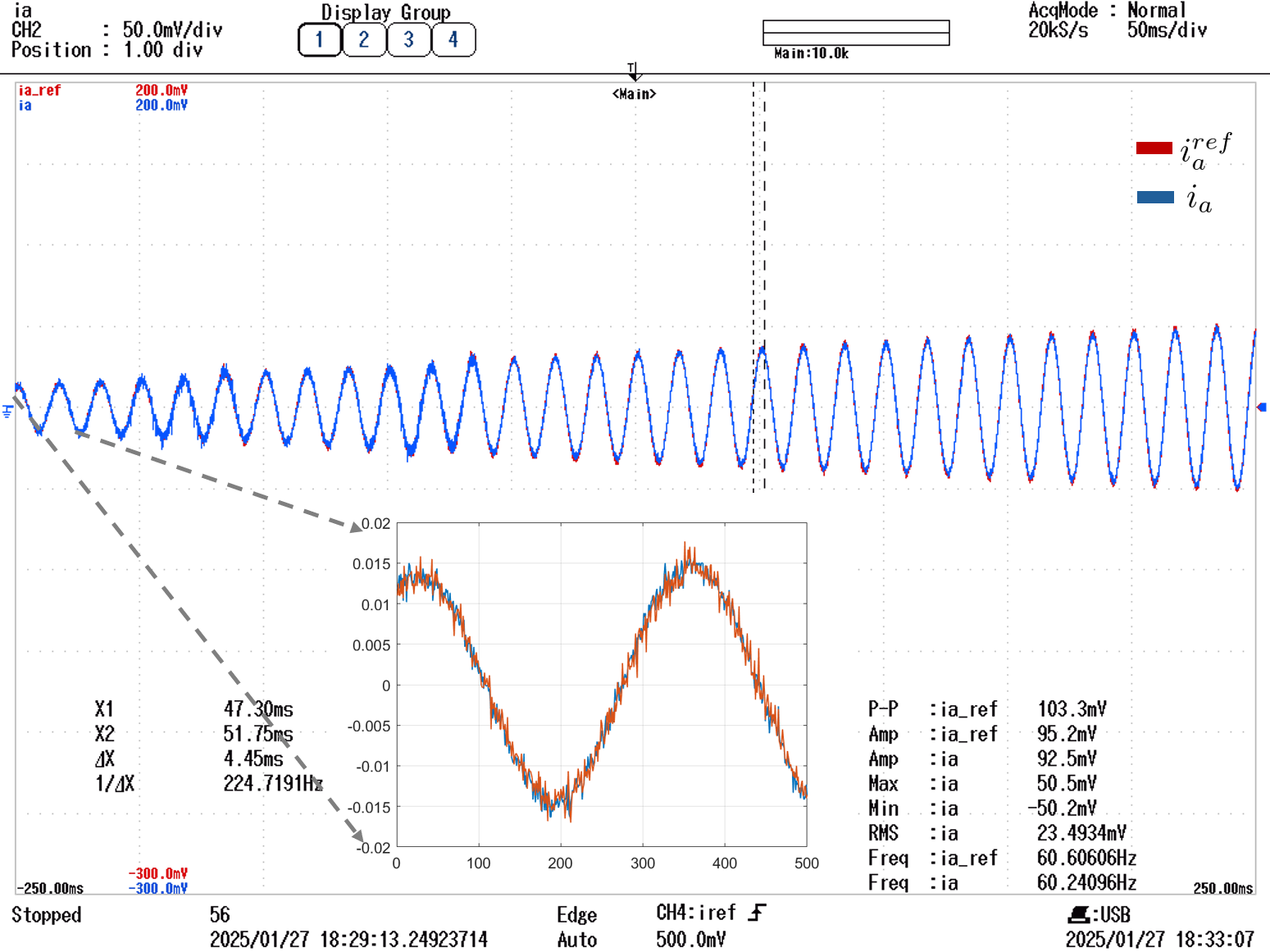}
       \caption{Normalized current tracking response for phase A with reference from voltage control loop.}
    \label{current_track_experiment}
\end{figure}

\begin{figure}[!htpb]
    \centering
\includegraphics[scale=0.65]{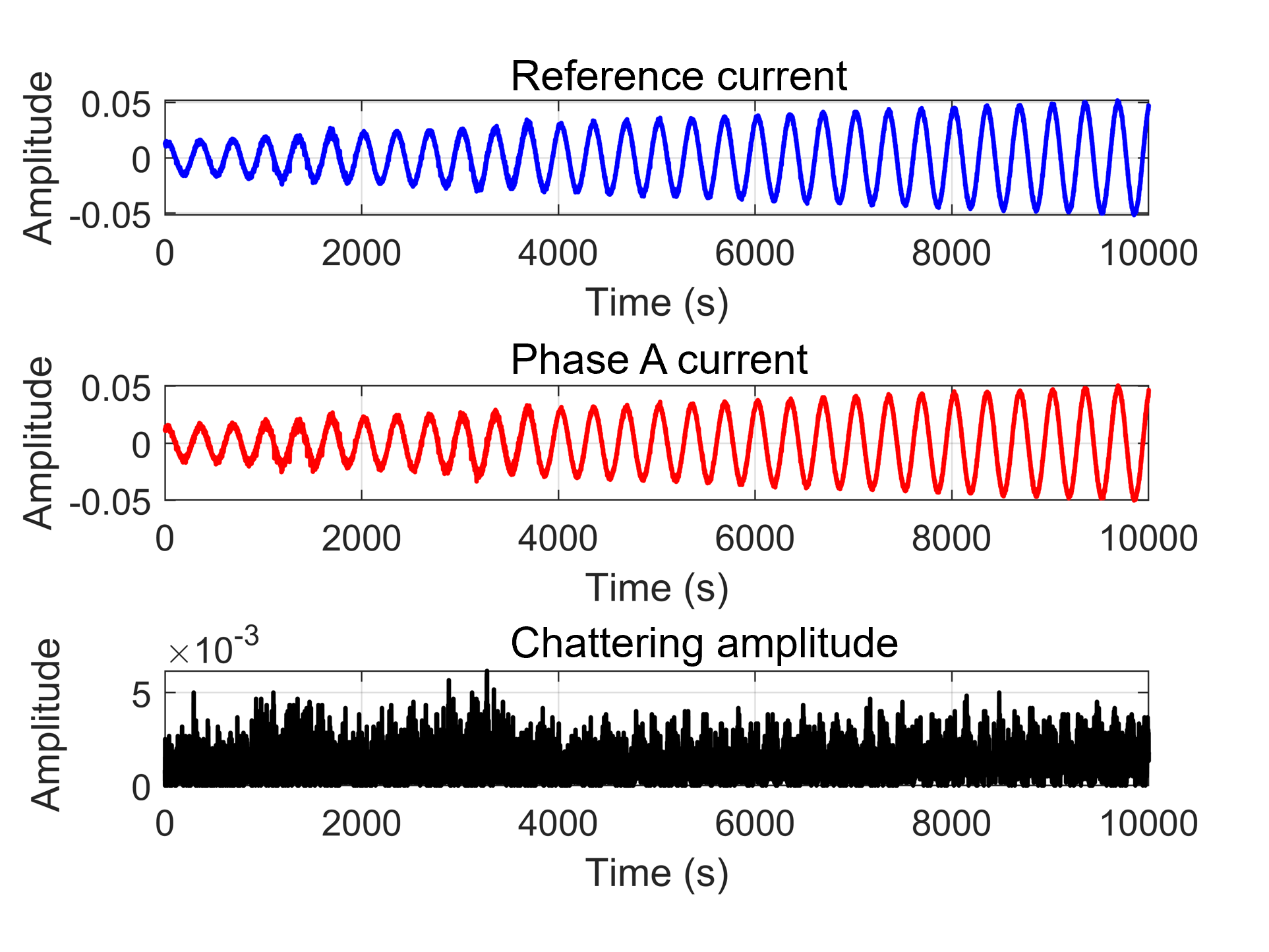}
       \caption{Reference current, measured Phase A current, and corresponding chattering amplitude under the proposed control scheme.}
      \label{phase_a_chattering_resp}
    \label{phase_a}
\end{figure}

\section{Conclusion}
This paper proposes a disturbance-adaptive non-singular finite-time control design for the regulation and stabilization of DC voltage and AC current in a rectifier system. Detailed stability analysis is provided, along with conditions for selecting appropriate controller gains. Numerical simulations and experimental results validate the effectiveness of the proposed methods, and demonstrates that the proposed controllers outperforms baseline approaches evaluated in this study in terms of reaching time, and steady-state error minimization. Based on the demonstrated results, the proposed controllers exhibit the potential for deployment in industrial electrical and electronic applications requiring accurate and fast AC voltage and current rectification. Additionally, they are well-suited for integration into consumer devices, where reliable and efficient rectification is essential to ensure optimal performance and energy efficiency. Future work will extend the design to incorporate grid-tied back-to-back converters. 

\section*{Acknowledgment}
This work was supported by the U.S. Department of Energy (Award DE-CR0000028).

\section*{Declaration of competing interests}
Conflict of interest – none declared.

\bibliographystyle{IEEEtran}
\bibliography{myreferences}

\vspace{0.5in}
\textbf{Appendix}
\begin{proof}[Proof of Lemma 1]
From \eqref{d1}--\eqref{d2}, differentiating $y = \eta_f + \sigma z$ gives
$\dot{y} + \sigma y = \sigma \dot{z}$.
Defining the derivative filter error $e = y - \dot{z}$ and integrating, we obtain
\[
e(t) = e^{-\sigma t} e(0) + e^{-\sigma t} \int_0^t \ddot{z}(\tau) e^{\sigma \tau} d\tau.
\]
Since $|\ddot{z}|\le \epsilon$, this implies $
|e(t)| \le e^{-\sigma t}|e(0)| + \frac{\epsilon}{\sigma},
$
so that $\lim_{t\to\infty}|e(t)| \le \frac{\epsilon}{\sigma}$.
\end{proof}

\begin{proof}[Proof of Lemma 2]
Let $s_i$ denote the $i$-th component of $\mathbf{s} \in \mathbb{R}^n$, and let $\epsilon > 0$. Then, we have  
\begin{align*}
    \mathbf{s}^\top \textsf{sat}\left( \frac{{\mathbf{s}}}{\epsilon} \right) &=  
    \sum_{i=1}^{n} \left\{  
    \begin{array}{cc}  
        \frac{{s}_i^2}{\epsilon}, & |{s}_i| < \epsilon \\  
        |{s}_i|, & \text{otherwise}  
    \end{array}  
    \right. \\  
    &\geq \sum_{i=1}^{n} \min \left\{ \frac{{s}_i^2}{\epsilon}, |{s}_i| \right\} \\  
    &\geq \min \left\{ \frac{\|\mathbf{s}\|_2^2}{\epsilon}, \|\mathbf{s}\|_1 \right\} \\  
    &\geq \min \left\{ \frac{\|\mathbf{s}\|_2^2}{\epsilon}, \sqrt{n} \|\mathbf{s}\|_2 \right\} \\  
    &\geq \sqrt{n} \|\mathbf{s}\|_2, \quad \text{for all} 
  \hspace{1mm} \|\mathbf{s}\|_2 \leq \epsilon \sqrt{n}.  
\end{align*}  
\end{proof}           

\end{document}